\documentclass[a4paper,oneside, reqno]{amsart}
\usepackage{mathrsfs}
\usepackage{amsfonts}
\usepackage{amssymb}
\usepackage{amsxtra}
\usepackage{dsfont}
\usepackage{amsthm}
\usepackage[colorlinks=true,linkcolor=blue]{hyperref}%
\usepackage[T1]{fontenc}
\usepackage[utf8]{inputenc}
\usepackage[english]{babel}
\usepackage{amsmath} 
\usepackage{bm}
\usepackage{graphicx}
\usepackage{hyperref}
\usepackage{color}

\usepackage{hyperref}
\usepackage{stmaryrd}
\usepackage{listings}

\usepackage{enumitem}

\numberwithin{equation}{section}
\usepackage{url}
\newtheorem{Thm}{Theorem}[section]

\newtheorem{Cor}[Thm]{Corollary}
\newtheorem{Lem}[Thm]{Lemma}
\theoremstyle{definition}
\newtheorem{Rem}[Thm]{Remark}      
\theoremstyle{definition}
\newtheorem{Defn}[Thm]{Definition}
\newtheorem{Ex}[Thm]{Example}
\newtheorem{Fact}[Thm]{Fact}
\theoremstyle{definition}
\newtheorem*{Prob*}{Problem}
      
\theoremstyle{definition}

\newcommand{\Z}{\mathbb{Z}}
\newcommand{\D}{\mathbb{D}}
\newcommand{\kr}{\mathds{kr}}
\newcommand{\N}{\mathbb{N}}
\newcommand{\T}{\mathbb{T}}

\title[Discrete dyadic maximal function over $l^1$ balls: dimension-free estimates]
{Dimension-free estimates on $l^2 (\Z ^d)$ for discrete dyadic maximal function over $l^1$ balls: small scales}

\author{Jakub Niksiński}
\address[Jakub Niksiński]{Institute of Mathematics\\
	University of Wroclaw\\
	Plac Grunwaldzki 2\\
	50-384 Wrocław\\
	Poland}
\email{trolek1130@gmail.com}
\subjclass[2020]{42B15, 42B25}
\keywords{discrete maximal function, $l^1$ balls, dimension-free estimates}

\begin{document}

\selectlanguage{english}

\begin{abstract}
We give a dimension-free bound on $l^p(\Z ^d)$ for the discrete Hardy-Littlewood operator over the $l^1$ balls in $\Z ^d$ with small dyadic radii, where $p \in [2, \infty]$. 
\end{abstract}

\subjclass[2020]{42B15, 42B25}
\keywords{discrete maximal function, $l^1$ balls, dimension-free estimates}

\maketitle



\section{Introduction}
Let $G$ be a convex, bounded, closed symmetric subset of $\mathbb{R}^d$ with non-empty interior, we will call such $G$ a symmetric convex body. Natural examples are the $l^q$ balls:
\[
B^q= \Big\lbrace x\in \mathbb{R}^d : \| x \|_{q}=\Big( \sum_{i=1}^d |x_i|^q \Big)^{1/q} \leq 1 \Big\rbrace \ \text{for} \  q \in [1, \infty ), 
\]
\[
B^{\infty}= \Big\lbrace x \in \mathbb{R}^d : \|x \|_{\infty}= \max_{1 \leq i \leq d} |x_i| \leq 1 \Big\rbrace.
\]
With each symmetric convex body one can associate the corresponding Hardy-Littlewood averaging operator. For any $t>0$ and $x \in \mathbb{R}^d$ we define
\[
M_t^G f(x)= \frac{1}{|t \cdot G|} \int_{t \cdot G} f(x-y) \ dy, 
\]
for a locally integrable function $f$, where $t \cdot G= \lbrace tx : x \in G \rbrace$ and $|t \cdot G|$ denotes its Lebesgue measure. Now for any $p$ let $C_p(d,G)>0$ be the smallest number such that the following inequality
\[
 \| \sup_{t>0} |M_t^G f| \|_{L^p(\mathbb{R}^d)} \leq C_p(d,G) \| f \|_{L^p(\mathbb{R}^d)}
\]
holds for every $f \in L^p(\mathbb{R}^d)$. It is well known that $C_p(d,G)< \infty$ for all $p \in (1, \infty]$ and all symmetric convex bodies $G$. 
\par
In 1980s dependency of $C_p(d,G)$ on dimension $d$ has begun to be studied. Various results were obtained, but as of now the major conjecture in this topic is still open, namely that $C_p(d,G)$ can be bounded from the above by a number independent of dimension $d$ for each fixed $p \in (1, \infty]$. We recommend the survey article \cite{HL-cont} for a detailed exposition of the subject, which contains results that we skipped mentioning.
\par 
Similar questions can be considered for discrete analogue of the operator $M_t^G$. For every $t>0$ and every $x \in \Z ^d $ we define the discrete Hardy-Littlewood averaging operator
\[
\mathcal{M}_t^G= \frac{1}{| (t \cdot G) \cap \Z^d|} \sum_{ y \in t \cdot G \cap \Z^d} f(x-y), \ \ f \in l^1(\Z ^d),
\]
where $| (t \cdot G) \cap \Z^d|$ is the number of elements of the set $(t \cdot G) \cap \Z^d$. Similarly as before we define $\mathcal{C}_p(d,G)>0$ to be the smallest number such that the following inequality
\[
 \| \sup_{t>0} |\mathcal{M}_t^G f| \|_{l^p(\mathbb{Z}^d)} \leq \mathcal{C}_p(d,G) \| f \|_{l^p(\mathbb{Z}^d)}
\]
holds for every $f \in l^p( \Z ^d)$. Using similar methods as in the continuous case one can show that for any $p \in (1, \infty]$ and any symmetric convex body we have $\mathcal{C}_p(d,G) < \infty$. 
\par
What about dependency of $\mathcal{C}_p(d,G)$ on the dimension $d$?
We can ask similar questions as in the continuous setup, yet it turns out that the situation is much more delicate. Indeed in \cite{cubes} the authors constructed a family of ellipsoids $E(d) \subseteq \mathbb{R}^d$, with the property that for each $p \in (1, \infty]$ there exists $C_p>0$ such that for every $d \in \N$ we have
\[
\mathcal{C}_p(d, E(d)) \geq C_p ( \log d)^{1/p}.
\]
This means that if we want to establish dimension-free estimates for $\mathcal{C}_p(d, G)$ we need to restrict ourselves to specific sets $G$, which contain more symmetries; one of the simpler options are $B^q$ balls. \par
Literature in the discrete setting is not as fruitful, nevertheless there are some papers and  positive results in this regard, for example:
\begin{itemize}
    \item In \cite{cubes} it was proved that for every $p \in ( \frac{3}{2}, \infty ]$ there exists a constant $C_p>0$ such that for every $d \in \N$ and every $f \in l^p(\Z ^d)$ we have
    \[
    \| \sup_{t>0} |\mathcal{M}_t^{B^{\infty}} f| \|_{l^p(\Z ^d)} \leq C_p \| f \|_{l^p(\Z ^d)}.
    \]
    This result is as strong as the ones in continuous case. Unfortunately in the case of sets $B^q$ for $q \neq \infty$ authors of papers \cite{cubes}, \cite{Bourgain2020}, \cite{Bq}  could only obtain way weaker conclusions. 
    \item In \cite{Bourgain2020} it was proved that for every $p \in [2, \infty)$ there exists $C_p>0$ such that for every $d \in \N$ and every $f \in l^p(\Z ^d)$ we have
    \[
    \| \sup_{t \in \D} |\mathcal{M}_t^{B^2} f| \|_{l^p(\Z ^d)} \leq C_p \| f \|_{l^p(\Z ^d)},
    \]
    where $\D= \{ 2^n: n \in \N \}.$
    \item In \cite{Bq} by extending methods of \cite{Bourgain2020}  it was proved that for any $q \in (2, \infty)$ and any $p \in [2, \infty)$ there exists constant $C(p,q)>0$ 
    such that for every $d \in \N$ and every $f \in l^p(\Z ^d)$ we have
    \[
    \| \sup_{t \in \D, t \geq d^{1/q}} |\mathcal{M}_t^{B^q} f| \|_{l^p(\Z ^d)} \leq C(p,q) \| f \|_{l^p(\Z ^d)},
    \]
    where $\D= \{ 2^n: n \in \N \}.$ The paper \cite{Bq} did not cover the range $t<d^{1/q}$ nor $q<2$.
\end{itemize}
\par
In this paper we will prove the following result.
\begin{Thm} \label{thm:1.1}
For any $p \in [2,\infty)$, $d \in \mathbb{N}$, $d \geq 4$ and any $f \in l^p(\Z ^d)$ we have
\[ \| \sup_{t \in \mathbb{D}, \ t \leq \sqrt{d}} |\mathcal{M}_t^{B^1} f| \|_{l^p(\mathbb{Z}^d)}  \leq 40^{2/p} \| f \|_{l^p(\Z ^d)}, \]
where $\mathbb{D}=\lbrace 2^n : n \in \mathbb{N}_0 \rbrace$ is the set of dyadic integers.
\end{Thm}
As far as the author is aware the result above is new. It is a small step beyond the results of \cite{Bq} in the case $q=1$. Additionally we give explicit values to various constants. In the proof of Theorem \ref{thm:1.1} we will rely on methods shown in the \cite[Section 3]{Bourgain2020}. Right at the beginning we will replace $t \cdot B^1$ by the simpler set $S_t=\lbrace x \in \lbrace -1,0,1 \rbrace ^d: \sum_{i=1}^d |x_i|=t \rbrace$ and then consider only $S_t$. Later on we merely use tools from \cite[Section 3]{Bourgain2020}.
This replacement by a simpler set is easy due to the fact that there exists a formula for the number of lattice points in $t \cdot B^1$. This is not the case for other sets of the form $t \cdot B^q$, where $q \in (1, \infty)$. We restrict ourselves to $t \leq \sqrt{d}$ due to the fact that then $(t \cdot B^1) \cap \Z ^d$ consists mostly of the points from $S_t$ (see Lemma \ref{lem:2.4}). This is not the case when $t$ is sufficiently big in terms of $d$ and we would need different argument in that range.
Throughout the rest of this paper we will consider $B^q$ only for $q=1$.
\subsection{Notation}
\begin{enumerate}
    \item For $t>0$ we define $B_t= \lbrace (x_1,...,x_d) \in \mathbb{R}^d : \sum_{i=1}^d |x_i| \leq t \rbrace$. 
    \item $\mathbb{N}=\lbrace 1,2,... \rbrace$ will denote the set of positive integers and $\mathbb{N}_0=\mathbb{N} \cup \lbrace 0 \rbrace$ will denote the set of non-negative integers. $\mathbb{D}=\lbrace 2^n : n \in \mathbb{N}_0 \rbrace$ is the set of dyadic integers. We define
$\mathbb{N}_N= \lbrace 1,2,...,N \rbrace $. 
\item We define $e(x)=e^{2\pi i x}$ for any $x \in \mathbb{R}^d$. 
\item We have standard scalar product on $\mathbb{R}^d$
\[
 x \cdot y= \sum_{k=1}^d x_ky_k,
\]
where $x,y \in \mathbb{R}^d.$ 
\item For $f,g \in l^2(\mathbb{Z}^d)$ we define $f \ast g \in l^2(\mathbb{Z}^d)$ by the series
\[f \ast g(x)= \sum_{y \in \mathbb{Z}^d } f(y)g(x-y)=\sum_{y \in \mathbb{Z}^d } f(x-y)g(y),\]
which is absolutely convergent for each $x \in \mathbb{Z}^d$.
\item  If $f \in l^1(\Z ^d)$ we define the discrete Fourier transform by the formula
\[
 \widehat{f}(\xi)= \sum_{x \in \Z ^d} f(x) e( x \cdot \xi ), \ \text{ for } \ \xi \in \T^d. 
\]
One can extend the discrete Fourier transform to $f \in l^2(\Z ^d)$. Then it turns out that $\widehat{f} \in L^2(\T^d)$. We have the following Parseval identity:
\[
  \|\widehat{f}\|_{L^2(\T^d)}=\| f \|_{l^2(\Z^d)}.
\]
Moreover, for any $f,g \in l^2(\Z ^d),$
\[
 \widehat{f \ast g}(\xi)= \widehat{f}(\xi) \widehat{f}(\xi).
\]
$\mathcal{F}^{-1}$ will denote the inverse of the discrete Fourier transform, that is
\[
\mathcal{F}^{-1}(G)(x)= \int_{\T^d} G(\xi) e(-x \cdot \xi) \ d \xi,
\]
where $G \in L^2(\T ^d)$.
\item We let $m_t$ be the multiplier symbol \[m_t(\xi)=\frac{1}{|B_t \cap \mathbb{Z}^d|}  \sum_{x \in B_t \cap \mathbb{Z}^d} e(x \cdot \xi). \]
\item $\mathbb{T}^d$ will denote the $d$-dimensional torus, which will be identified with the set $[-\frac{1}{2}, \frac{1}{2})^d.$ \\
\item $Sym(d)$ will denote the permutation group of $\lbrace 1,2,...,d \rbrace$.
 \end{enumerate}
\begin{Defn}  \label{def:1.2} The discrete Hardy-Littlewood averaging operator of the $l^1$ ball is defined for any function $f:\mathbb{Z}^d \to \mathbb{C}$ by the formula
 \[ \mathcal{M}_t f(x)=\frac{1}{|B_t \cap \mathbb{Z} ^d |} \sum_{ y \in B_t \cap \mathbb{Z} ^d } f(x-y). \]

\end{Defn}
Our goal is to prove the following theorem.
\begin{Thm} \label{thm:1.3}
For any $d \in \mathbb{N}$, $d \geq 4$ and any $f \in l^2(\Z ^d)$ we have
\[ \| \sup_{t \in \mathbb{D}, \ t \leq \sqrt{d}} |\mathcal{M}_t f| \|_{l^2(\mathbb{Z}^d)}  \leq 40 \| f \|_{l^2(\Z ^d)}. \]
\end{Thm}
Using a complex interpolation argument one can show that Theorem $\ref{thm:1.3}$ implies Theorem $\ref{thm:1.1}$. 
Notice that the operator $\mathcal{M}_t$ is given by convolution, that is
\[ \mathcal{M}_t f(x)=f \ast K_t(x), \]
where 
\[ K_t(y)= \frac{1}{|B_t \cap \mathbb{Z}^d|} \mathds{1}_{B_t \cap \mathbb{Z}^d}(y). \]
Since we are considering $l^2$ norms, using Fourier theory we can reduce our problem to understanding pointwise bounds of the following function 
\[ m_t(\xi)=\widehat{K_t}(\xi)= \frac{1}{|B_t \cap \mathbb{Z}^d|} \sum_{x \in B_t \cap \mathbb{Z}^d} e(x \cdot \xi).\]
This will be explained in detail in the next sections. 

\section{Estimate for $|B_t \cap \mathbb{Z}^d|$ and its consequences.}
The goal of this section is to approximate function $m_t(\xi)$ by a simpler function; it will be obtained by approximating $|B_t \cap \mathbb{Z}^d|$  by the number of lattice points in a simpler set.
\begin{Defn} \label{def:2.1}
For $t \in \mathbb{N}$ we define 
\[ S_t= \lbrace x \in \lbrace -1,0,1 \rbrace ^d: \sum_{i=1}^d |x_i|=t \rbrace, \]
moreover we define function $s_t: \mathbb {T}^d \to \mathbb{C}$  by formula
\[ s_t(\xi)= \frac{1}{|S_t|} \sum_{x \in S_t} e(x \cdot \xi),\]
next we introduce averaging operator $\mathcal{S}_t$ related to the above multiplier by the formula
\[\mathcal{S}_tf(x)=\mathcal{F}^{-1}(s_t \hat{f}) (x), \ \text{for any} f \in l^2(\mathbb{Z}^d).\]
\end{Defn}
Notice that $S_t \subseteq B_t \cap \mathbb{Z}^d$.
\begin{Lem} \label{lem:2.2}
For any $t \in \mathbb{N}, t \leq d$ the following formulas hold
\begin{equation}
    |S_t|=2^t \binom{d}{t} \label{eq:2.1}
\end{equation}
    \begin{equation}
    |B_t \cap \mathbb{Z}^d|=\sum_{l=0}^t 2^l \binom{d}{l} \binom{t}{l} \label{eq:2.2}.
\end{equation}
\end{Lem}
These formulas are well known. We give a proof for completeness.
\begin{proof}
For equality \eqref{eq:2.1} we notice that there are $\binom{d}{t}$ options to choose $t$ coordinates out of $d$ coordinates, on which we have value $-1$ or $1$. Next we can choose each sign in $2$ ways, thus we obtain $|S_t|=2^t \binom{d}{t}.$ \\
For equality \eqref{eq:2.2} we consider the following disjoint decomposition 
\[ B_t \cap \mathbb{Z}^d= \bigcup_{k=0}^t A_k,\]
where 
\[ A_k= \lbrace (x_1,...,x_d) \in \mathbb{Z}^d: \sum_{i=1}^d |x_i|=k\rbrace. \]
Now for each $k \in \lbrace 0,1,...,t \rbrace$ we will compute $|A_k|$. We have $|A_0|=1$. Consider $k \geq 1$. Then 
\[ A_k= \bigcup_{l=1}^k \lbrace \textbf{x} \in \mathbb{Z}^d: \sum_{i=1}^d |x_i|=k, \ \textbf{x} \text{ has $l$ exactly nonzero coordinates } \rbrace.  \]
In the union above each $l$-th individual set has size
\[ 2^l \binom{d}{l} \cdot | \lbrace (y_1,...,y_l) \in \mathbb{N}^l : y_1+...+y_l=k \rbrace |=  2^l \binom{d}{l} \cdot \binom{k-1}{l-1} .\]
Indeed, we have $\binom{d}{l}$ options of choosing nonzero coordinates and then $2^l$ options of choosing signs. The last equality above is classical "stars and bars" problem from combinatorics, see \cite[p. 38]{feller1950introduction}. Hence for $k \geq 1$
\[ |A_k|= \sum_{l=1}^k 2^l \binom{d}{l} \cdot \binom{k-1}{l-1}, \]
from which it follows
\[ |B_t \cap \mathbb{Z}^d|= \sum_{k=0}^t |A_k|=1+\sum_{k=1}^t  \sum_{l=1}^k 2^l \binom{d}{l} \cdot \binom{k-1}{l-1} \]
\[ =1+ \sum_{l=1}^t 2^l \binom{d}{l} \sum_{k=l}^t \binom{k-1}{l-1} =1+ \sum_{l=1}^t 2^l \binom{d}{l} \binom{t}{l} =\sum_{l=0}^t 2^l \binom{d}{l} \binom{t}{l}. \]
In the penultimate equality above we used the following fact: for any $t \geq l \geq 1$ we have
\[ \binom{t}{l}= \sum_{k=l}^t \binom{k-1}{l-1}.\]
One can prove this equality by segregating $l$-element subsets of the set $\lbrace 1,...,t \rbrace$ based on biggest number which they contain. For each $k$ there are $\binom{k-1}{l-1}$ subsets of $\lbrace 1,...,t \rbrace$ with  $l$ elements,  whose biggest number is $k$. 
\end{proof}
\begin{Rem} \label{rem:2.3}
Notice that $|S_t|$ is equal to term corresponding to $k=t$ in the sum $\eqref{eq:2.2}$. When $t$ is small in terms of $d$ then $\binom{d}{k+1}$ is roughly $d$ times bigger than $\binom{d}{k}$ for each $k \leq t$; because of that we expect $|B_t \cap \mathbb{Z}^d|$ to be similar in size to $|S_t|$ when $t$ is sufficiently small in terms of $d$. Lemma \ref{lem:2.4} quantifies our expectations.
\end{Rem}
\begin{Lem} \label{lem:2.4}
Let $t,d \in \mathbb{N}$ be numbers such that $d \geq 4$, $1 \leq t \leq \sqrt{d}$. Then the following inequalities hold
\[\frac{t^2}{2d} \leq \frac{|B_t \cap \mathbb{Z}^d|-|S_t|}{|S_t|} \leq e \frac{t^2}{d}.\]
\end{Lem}
\begin{proof}
Assumptions $d \geq 4$ and $t \leq \sqrt{d}$ imply that $2t \leq d$. \\
Using formulas from Lemma \ref{lem:2.2} we get that
\[\frac{|B_t \cap \mathbb{Z}^d|-|S_t|}{|S_t|} = \frac{\sum_{k=0}^{t-1} 2^k \binom{t}{k} \binom{d}{k}}{2^t \binom{d}{t}}= \sum_{k=0}^{t-1} 2^{k-t} \binom{t}{k} \frac{t!}{k!} \frac{(d-t)!}{(d-k)!}\]
\[ =\sum_{k=0}^{t-1} 2^{k-t} \binom{t}{k} \frac{t!}{k!} \frac{1}{(d-k)(d-k-1)...(d-t+1)}<\sum_{k=0}^{t-1} 2^{k-t} \binom{t}{k} \frac{t!}{k!} (d-t)^{k-t} \] 
\[= \sum_{k=0}^{t-1} 2^{k-t} \binom{t}{k}^2 (t-k)! (d-t)^{k-t} = \sum_{u=1}^{t} (2(d-t))^{-u} \binom{t}{u}^2 u!= \sum_{u=1}^{t} a_u.\]
In the penultimate equality above we changed summation index by substitution $u=t-k$ and used the fact that $\binom{t}{k}=\binom{t}{t-k}$, in the last equality we denoted each term of the sum by $a_u$. \\
Then for each $1 \leq u \leq t-1$ we have that $a_{u+1} \leq \frac{t^2}{d} \frac{1}{u+1} a_{u}$. Indeed, 
\[ a_{u+1}= (2(d-t))^{-u-1}\binom{t}{u+1}^2 (u+1)!=(2(d-t))^{-u-1}\frac{t!^2}{(u+1)! (t-u-1)!^2} \]
\[ =(2(d-t))^{-u-1} \frac{t!^2}{u!(t-u)!^2} \cdot \frac{(t-u)^2}{u+1}=  \frac{(t-u)^2}{2(d-t)(u+1)} \cdot (2(d-t))^{-u} \binom{t}{u}^2 u! \] 
\[ =\frac{(t-u)^2}{2(d-t)(u+1)} a_u \leq  \frac{t^2}{2(d-t)(u+1)} a_u  \leq \frac{t^2}{d} \frac{1}{u+1}a_u.
  \]
  Iterating the above estimate we get
 \[ a_{u+1} \leq \frac{t^2}{d} \frac{1}{u+1}a_u \leq \Big(\frac{t^2}{d} \Big)^2 \frac{1}{(u+1) \cdot u}a_{u-1} \leq ... \]
 \[
 \leq \Big(\frac{t^2}{d} \Big)^u \cdot \frac{1}{(u+1) \cdot u \cdot... \cdot 2} a_1= \Big(\frac{t^2}{d} \Big)^u \cdot \frac{1}{(u+1)!} a_1. \]
 However, using the fact that $2t \leq d$ we obtain that
 \[ a_1= \frac{t^2}{2(d-t)} \leq \frac{t^2}{d},\]
 which implies
 \[a_{u+1} \leq \Big(\frac{t^2}{d} \Big)^{u+1} \cdot \frac{1}{(u+1)!}\]
 for any $0 \leq u \leq t-1$, hence by $t \leq \sqrt{d}$ we get
 \[\sum_{u=1}^{t} a_u \leq \sum_{u=1}^{t} \frac{\Big(\frac{t^2}{d} \Big)^u}{u!} <  e \frac{t^2}{d}.  \]
 Thus we proved the upper bound. \\
 On the other hand the lower bound follows from estimating undermentioned sum by the term $k=t-1$.
\[ \frac{|B_t \cap \mathbb{Z}^d|- |S_t| }{|S_t|}= \frac{\sum_{k=0}^{t-1} 2^k \binom{t}{k} \binom{d}{k}}{2^t \binom{d}{t}}= \sum_{k=0}^{t-1} 2^{k-t} \binom{t}{k} \frac{t!}{k!} \frac{(d-t)!}{(d-k)!} \]
\[> \frac{1}{2}\binom{t}{t-1} \frac{t!}{(t-1)!} \frac{(d-t)!}{(d-t+1)!}= \frac{t^2}{2(d-t+1)} \geq \frac{t^2}{2d}. \]
\end{proof}
In the paper we will use only the upper bound of Lemma \ref{lem:2.4}. Its main application is given in the Corollary \ref{cor:2.5} below.
\begin{Cor} \label{cor:2.5} 
For $d,t \in \mathbb{N} $ such that $d \geq 4$, $ 1 \leq t \leq \sqrt{d}$ and any $\xi \in \mathbb{T}^d$ the following inequality holds
\[|m_t(\xi)- s_t(\xi)| \leq 2e \frac{t^2}{d}.\]
\end{Cor}
\begin{proof}
Using the disjoint decomposition $B_t\cap \mathbb{Z}^d=S_t \cup \Big( B_t\cap \mathbb{Z}^d \cap S_t \Big)$ and the upper bound from Lemma \ref{lem:2.4} we obtain that
\[ |m_t(\xi)- s_t(\xi)|= \Big| \frac{1}{|B_t \cap \mathbb{Z}^d|} \sum_{x \in B_t \cap \mathbb{Z}^d} e(x \cdot \xi) - \frac{1}{|S_t|} \sum_{x \in S_t} e(x \cdot \xi)  \Big| \]
\[ \leq  \Big| \sum_{x \in S_t} e(x \cdot \xi)  \Big| \cdot \Big( \frac{1}{|S_t|}- \frac{1}{|B_t \cap \mathbb{Z}^d|} \Big) + \frac{1}{|B_t \cap \mathbb{Z}^d|} \Big| \sum_{x \in B_t \cap \mathbb{Z}^d \setminus S_t} e(x \cdot \xi) \Big| \]
\[ \leq  2 \frac{|B_t \cap \mathbb{Z}^d|-|S_t|}{|B_t \cap \mathbb{Z}^d|}=2 \frac{|S_t|}{|B_t \cap \mathbb{Z}^d|} \cdot \frac{|B_t \cap \mathbb{Z}^d|-|S_t|}{|S_t|} \leq 2e \frac{t^2}{d}\]
\end{proof}
The simple bound above has consequences in terms of norm estimates for the corresponding maximal dyadic operator. Proof of the next lemma is based on Lemma \ref{lem:2.4} and a square function estimate.

\begin{Lem} \label{lem:2.6}
For $d \in \mathbb{N}$, $d \geq 4$ and any $f \in l^2(\Z ^d)$ the following inequality holds
 \[  \| \sup_{t \in \mathbb{D}, t \leq \sqrt{d}} |(\mathcal{M}_t-\mathcal{S}_t)f| \|_{l^2(\mathbb{Z}^d)} \leq 6 \| f \|_{l^2(\Z ^d)} .\]
\end{Lem}
\begin{proof}
Take any $f \in l^2(\Z^d)$, then we have
\[ \|  \sup_{t \in \mathbb{D}, t \leq \sqrt{d}} |(\mathcal{M}_t-\mathcal{S}_t)f| \|_{l^2(\Z^d)} \leq \| \Big( \sum_{t \in \D, t \leq \sqrt{d}} |(\mathcal{M}_t-\mathcal{S}_t)f|^2 \Big)^{1/2} \|_{l^2(\Z^d)} \]
\[ = \Big( \sum_{x \in \Z^d} \sum_{t \in \D, t \leq \sqrt{d}} |(\mathcal{M}_t-\mathcal{S}_t)f(x)|^2 \Big)^{1/2}= \Big( \sum_{t \in \D, t \leq \sqrt{d}} \sum_{\ x \in \Z^d}  |(\mathcal{M}_t-\mathcal{S}_t)f(x)|^2 \Big)^{1/2} \]
\[  = \Big( \sum_{t \in \D, t \leq \sqrt{d}} \|  |(\mathcal{M}_t-\mathcal{S}_t)f| \|^2_{l^2(\Z^d)} \Big)^{1/2}=\Big( \sum_{t \in \mathbb{D}, t \leq \sqrt{d}} \int_{\mathbb{T}^d} |m_t(\xi) - s_t(\xi)|^2 | \widehat{f}|^2 d \xi \Big)^{1/2}  \]
\[ \leq 2e \Big( \sum_{t \in \D, t \leq \sqrt{d}} \frac{t^4}{d^2} \int_{\mathbb{T}^d} | \widehat{f}|^2 d\xi \Big)^{1/2}=  \| \widehat{f} \|_{L^2(\mathbb{T}^d)} \cdot \frac{2e}{d} \Big( \sum_{n=1}^{\lfloor \log_2(\sqrt{d}) \rfloor} 2^{4n} \Big)^{1/2}  \]
\[  =\| \widehat{f} \|_{L^2(\mathbb{T}^d)} \cdot \frac{2e}{d} \Big(2^4 \cdot \frac{2^{4\lfloor \log_2(\sqrt{d}) \rfloor}-1}{2^4-1} \Big)^{1/2} \leq \| \widehat{f} \|_{L^2(\mathbb{T}^d)} \cdot \frac{2e}{d} \Big(\frac{16}{15} \cdot d^2 \Big)^{1/2}
\]
\[ \leq  6\| \widehat{f} \|_{L^2(\mathbb{T}^d)}=6 \| f\|_{l^2(\Z^d)}.\]
The second inequality above comes from Corollary \ref{cor:2.5}.
In the above reasoning multiple times we made use of Parseval identity and the fact that $\mathcal{M}_t(f)=\mathcal{F}^{-1}(m_t \cdot \widehat{f})$, $\mathcal{S}_t(f)=\mathcal{F}^{-1}(s_t \cdot \widehat{f})$.
\end{proof}
\section{Estimates for the multiplier $s_t$.}
In the proof of Lemma $\ref{lem:2.6}$ we have seen how a pointwise bound of $|m_t(\xi)-s_t(\xi)|$ impacted \\
$\| \sup_{t \in \D, t \leq \sqrt{d}} |(\mathcal{M}_t-\mathcal{S}_t )f|\|_{l^2(\Z^d)}$. The bound in Corollary \ref{cor:2.5} is fairly strong and uniform in $\xi \in \T ^d$. Unfortunately we will not be that lucky in the future.
Our goal for now is to obtain bounds on $|s_t(\xi)|$ for big $\xi$ (then we expect some decay in $\xi$) and $|s_t(\xi)-1|$ for small $\xi$. \\
Proof of the next lemma exploits in a simple way  invariance of the set $S_t$ under both permutation of coordinates and sign changes of each coordinate.
\begin{Lem}
For every $d,t \in \mathbb{N}$, $t \leq d$ and every $\xi \in \T^d$ we have
\[  |s_t(\xi)-1| \leq 2 \frac{t}{d} \sum_{j=1}^d \sin^2(\pi \xi_j). \] \label{lem:3.1}
\end{Lem}
\begin{proof}
Notice that for any $\epsilon \in \lbrace -1,1 \rbrace^d$ we have that 
\[ x \in S_t \iff \epsilon  x \in S_t,\]
where $\epsilon x=(\epsilon_1 x_1,..., \epsilon_d x_d)$.
This implies that
\[s_t(\xi)= \frac{1}{|S_t|} \sum_{ \epsilon  x \in S_t} e(\epsilon  x \cdot \xi)= \frac{1}{|S_t|} \sum_{ x \in S_t} e( \epsilon  x \cdot \xi). \]
Hence
\[s_t(\xi)= \frac{1}{2^d} \sum_{\epsilon \in \lbrace -1,1 \rbrace^d} s_t(\xi)= \frac{1}{2^d} \sum_{\epsilon \in \lbrace -1,1 \rbrace^d} \frac{1}{|S_t|} \sum_{ x \in S_t} e( \epsilon  x \cdot \xi)
\]
\[=\frac{1}{|S_t|} \sum_{ x \in S_t} \frac{1}{2^d} \sum_{\epsilon \in \lbrace -1,1 \rbrace^d} e( \epsilon  x \cdot \xi)= \frac{1}{|S_t|} \sum_{ x \in S_t} \prod_{j=1}^d \cos(2 \pi x_j \xi_j). \]
Note that for any sequences of complex numbers $\lbrace a_j \rbrace_{j=1}^d, \lbrace b_j \rbrace_{j=1}^d $ such that $\max_{1 \leq j \leq d} |a_j| \leq 1$ and $\max_{1 \leq j \leq d} |b_j| \leq 1$ we have
\begin{equation}   
\Big| \prod_{j=1}^d a_j - \prod_{j=1}^d b_j \Big| \leq \sum_{j=1}^d |a_j-b_j|, \label{eq:3.1}
\end{equation}

This follows from an easy induction argument. Using \eqref{eq:3.1} and the formula $\cos(2x)=1-2\sin^2(x)$ we obtain
\[
|s_t(\xi)-1|  \leq \frac{1}{|S_t|} \sum_{ x \in S_t} \Big| \prod_{j=1}^d \cos(2 \pi x_j \xi_j)-1 \Big| \leq \frac{1}{|S_t|} \sum_{ x \in S_t} \sum_{j=1}^d |\cos(2 \pi x_j \xi_j)-1| \]
\[= \frac{2}{|S_t|} \sum_{ x \in S_t} \sum_{j=1}^d \sin^2(\pi x_j \xi_j) \leq  \frac{2}{|S_t|} \sum_{j=1}^d \sin^2(\pi\xi_j)  \cdot \sum_{ x \in S_t} x_j^2.\]
In the last inequality we used the fact that $|\sin(xy)| \leq |x| |\sin(y)|$ for any $x \in \mathbb{Z}, y \in \mathbb{R}$ (this can be proved by only considering $x \in \N$ and induction with respect to $x$ using formula for $\sin(a+b)$).
\\
Now to understand the sum $\sum_{ x \in S_t} x_j^2$ we make use of the fact that $S_t$ is closed under permutations of coordinates, that is for any $\sigma \in Sym(d)$ we have
\[  x \in S_t \iff \sigma (x) \in S_t,  \]
where $\sigma ( x)=(x_{\sigma(1)},x_{\sigma(2)},...,x_{\sigma(d)})$.
Using similar argument as in the beginning of the proof one can show that for any $j \in \lbrace 1,...,d \rbrace$ we have
\[ \sum_{x \in S_t} x_j^2= \frac{1}{d} \sum_{x \in S_t} \| x \|_{l^2}^2.\]
Thus
\[|s_t(\xi)-1| \leq \frac{2}{|S_t|} \sum_{j=1}^d \sin^2(\pi\xi_j)  \cdot \sum_{ x \in S_t} x_j^2= \frac{2}{|S_t|} \sum_{j=1}^d \sin^2(\pi\xi_j)  \cdot \frac{1}{d} \sum_{x \in S_t} \| x \|_{l^2}^2  \]
\[= \frac{2}{d} \sum_{j=1}^d \sin^2(\pi\xi_j)  \cdot \frac{1}{|S_t|} \sum_{x \in S_t} \| x \|_{l^1}= 2 \frac{t}{d} \sum_{j=1}^d \sin^2(\pi \xi_j).
\]
Above we used the fact that any for any $x \in S_t$ we have $\| x \|_{l^2}^2= \| x \|_{l^1} $ (since $x$ has coordinates in the set $\lbrace-1,0,1 \rbrace$) and that  $\| x \|_{l^1}=t$.  
\end{proof}
\subsection{Krawtchouk polynomials}
\begin{Defn}
(Krawtchouk polynomial). For every $n \in \N _0$,
$k \in \lbrace 0,1,...,n \rbrace$ and $x \in \mathbb{R}$ we define $k$-th Krawtchouk polynomial by the formula
\[ \kr _k^{(n)}(x)= \frac{1}{\binom{n}{k}} \sum_{j=0}^k (-1)^j \binom{x}{j} \binom{n-x}{k-j}, 
\]
if $x$ is not integer, then  $\binom{x}{j}=\frac{x(x-1)...(x-j+1)}{j!}$. \label{def:3.2}
\end{Defn}
Next theorem describes important facts regarding Krawtchouk polynomials.
\begin{Thm}
For every $n \in \N _0$ and integers $x,k \in [0,n]$ we have
\begin{enumerate}
    \item{Symmetry:} $\kr _k^{(n)}(x)= \kr_x^{(n)}(k)$.
    \item{Reflection symmetry:} $\kr ^{(n)}_k(n-x)=(-1)^k \kr _k^{(n)}(x)  $.
    \item{Uniform bound:} \label{thm:3.3.5} For any $x,k \leq n/2$ we have
    \[|\kr ^{(n)}_k(x)| \leq e^{-ckx/n},\]
    where $c=-2 \log(0.93)=0.14514... $
\end{enumerate} \label{thm:3.3}
\end{Thm}
The proof of the first two points is contained in \cite{412678}, the last point is \cite[Lemma 2.2]{v010a003}. The value of the constant $c$ is not explicitly given in \cite{v010a003}, however one can deduce it from its proof.
From now on we define
\[
c=-2 \log(0.93)=0.14514...
\]
\begin{Lem}
For any $t,d \in \N _0$ such that $t \leq d/2$ we have
\begin{equation}
    |s_t(\xi)| \leq  \exp (- \frac{c t}{2d} \sum_{i=1}^d \sin( \pi \xi_i)^2)+ \exp (- \frac{c t}{2d} \sum_{i=1}^d \cos( \pi \xi_i)^2). \label{eq:3.2}
\end{equation}

\label{lem:3.4}
\end{Lem}
\begin{proof}
We will proceed as in the later part of the proof of \cite[Proposition 3.3]{Bourgain2020}.
Using the same symmetry invariance arguments as in the proof of Lemma \ref{lem:3.1} we have the following equalities
\[ s_t( \xi)= \frac{1}{|S_t|} \sum_{y \in S_t} \prod_{i=1}^d \cos(2 \pi y_i \xi_i)= \frac{1}{|S_t|} \sum_{y \in S_t}  \frac{1}{d!} \sum_{ \sigma \in Sym(d)} \prod_{i=1}^d \cos(2 \pi y_{\sigma(i)} \xi_i).\]
It is sufficient to prove that for any $y \in S_t$ we have that
\[\Big| \frac{1}{d!} \sum_{ \sigma \in Sym(d)} \prod_{i=1}^d \cos(2 \pi y_{\sigma(i)} \xi_i) \Big| \]
\[\leq   \exp (- \frac{c t}{2d} \sum_{i=1}^d \sin( \pi \xi_i)^2)+ \exp (- \frac{c t}{2d} \sum_{i=1}^d \cos( \pi \xi_i)^2) .\]
Fix any $y \in S_t$, notice that we have following disjoint decomposition.
\begin{equation}
    Sym(d)= \bigcup_{\substack{J \subseteq \N _d \\ |J|=t}} \lbrace \sigma \in Sym(d): |y_{\sigma(j)}|=1 \iff j \in J \rbrace.
    \label{eq:3.3}
\end{equation}
Notice that each set under the union has size $(d-t)! \cdot t!$.
Since cosine is even,  from \eqref{eq:3.3} we obtain
\[ \frac{1}{d!} \sum_{ \sigma \in Sym(d)} \prod_{i=1}^d \cos(2 \pi y_{\sigma(i)} \xi_i)= \frac{1}{\binom{d}{t}} \sum_{\substack{J \subseteq \N _d \\
|J|=t}} \prod_{j \in J} \cos(2 \pi \xi_j). \]
For any $S \subseteq \N _d$ we define 
\[ a_S(\xi)= \prod_{j \in \N _d \setminus S} \cos^2( \pi \xi_j) \cdot \prod_{i \in S} \sin^2(\pi \xi_i).\]
For any set $J \subseteq \N _d$ such that $|J|=t$ we define $\varepsilon(J) \in \lbrace -1,1 \rbrace^d$ to be vector, such that $\varepsilon(J)_j=-1$ exactly when $j \in J$. Then we have 
\[\prod_{j \in J}\cos(2\pi \xi_j)= \prod_{j \in \N _d} \Big( \frac{1+\cos(2 \pi \xi_j)}{2}+ \varepsilon(J)_j \frac{1- \cos(2 \pi \xi_J)}{2} \Big)\]
\[=\prod_{j \in \N _d} \big( \cos^2(\pi \xi_j)+ \varepsilon(J)_j \sin^2(\pi \xi_j) \big)= \sum_{S \subset \N _d} a_S(\xi) w_S(\varepsilon(J)),\]
where $w_S : \lbrace -1,1 \rbrace^d \to \lbrace -1,1 \rbrace$ is defined by the formula $w_S(\varepsilon)= \prod_{j \in S} \varepsilon_j$. Now by changing the order of summation we obtain
\[ \frac{1}{d!} \sum_{ \sigma \in Sym(d)} \prod_{i=1}^d \cos(2 \pi y_{\sigma(i)} \xi_i)=\sum_{S \subset \N _d} a_S(\xi) \frac{1}{\binom{d}{t}} \sum_{\substack{J \subseteq \N _d \\
|J|=t}} w_S(\varepsilon(J)). \]
Notice that for fixed $S \subseteq \N _d$ the following equalities hold
\[
\frac{1}{\binom{d}{t}} \sum_{\substack{J \subseteq \N _d \\
|J|=t}} w_S(\varepsilon(J))=\frac{1}{\binom{d}{t}} \sum_{\substack{J \subseteq \N _d \\
|J|=t}} (-1)^{|S \cap J|}= \frac{1}{\binom{d}{t}} \sum_{j=0}^t \sum_{\substack{J \subseteq \N _d \\
|J|=t \\
|S \cap J|=j}} (-1)^{j} \]
\[= \frac{1}{\binom{d}{t}} \sum_{j=0}^t (-1)^j \binom{|S|}{j} \binom{d-|S|}{t-j}= \kr _t^{(d)}(|S|).
\]
Using the above we get
\[
\frac{1}{d!} \sum_{ \sigma \in Sym(d)} \prod_{i=1}^d \cos(2 \pi y_{\sigma(i)} \xi_i)= \sum_{S \subset \N _d} a_S(\xi) \kr _t^{(d)}(|S|).
\]
If $|S| \leq d/2$ then by our assumption that $t \leq d/2$ and point  \ref{thm:3.3.5} of Theorem \ref{thm:3.3} we get
\[ |\kr_t^{d}(|S|)| \leq e^{- \frac{c|S|t}{d}},\]
otherwise if $|S|>d/2$ then using reflection symmetry and point  \ref{thm:3.3.5} of Theorem \ref{thm:3.3} we obtain
\[ | \kr_t^{d}(|S|)| =| \kr_t^{d}(d-|S|)| \leq e^{- \frac{c(d-|S|)t}{d}}.
\]
This gives us
\[
\frac{1}{d!} \Big| \sum_{ \sigma \in Sym(d)} \prod_{i=1}^d \cos(2 \pi y_{\sigma(i)} \xi_i) \Big| \leq \sum_{S \subseteq \N_d} a_S(\xi)e^{- \frac{c|S|t}{d}} +\sum_{S \subseteq \N_d} a_S(\xi)e^{- \frac{c(d-|S|)t}{d}}
\]
\[=
\sum_{S \subseteq \N_d} \prod_{j \in \N _d \setminus S} \cos^2( \pi \xi_j) \cdot \prod_{j \in S} e^{-\frac{ct}{d}}\sin^2(\pi \xi_j)
+
\sum_{S \subseteq \N_d} \prod_{j \in \N _d \setminus S} e^{-\frac{ct}{d}} \cos^2( \pi \xi_j) \cdot \prod_{j \in S} \sin^2(\pi \xi_j)
\]
\[=
\prod_{j \in \N _d} \big( \cos^2( \pi \xi_j)+e^{-\frac{ct}{d}}\sin^2(\pi \xi_j) \big)+ \prod_{j \in \N _d} \big( e^{-\frac{ct}{d}}\cos^2( \pi \xi_j)+\sin^2(\pi \xi_j) \big)
\]
\[
=\prod_{j \in \N _d} \big( 1-(1-e^{-\frac{ct}{d}})\sin^2(\pi \xi_j) \big)+ \prod_{j \in \N _d} \big(1-(1-e^{-\frac{ct}{d}})\cos^2(\pi \xi_j) \big) 
\]
\[ \leq
\exp\Big( -(1-e^{-\frac{ct}{d}}) \sum_{j=1}^d \sin^2(\pi \xi_j) \Big)+\exp\Big( -(1-e^{-\frac{ct}{d}}) \sum_{j=1}^d \cos^2(\pi \xi_j) \Big) 
\]
\[
\leq \exp\Big( -\frac{ct}{2d} \sum_{j=1}^d \sin^2(\pi \xi_j) \Big)+\exp\Big( -\frac{ct}{2d} \sum_{j=1}^d \cos^2(\pi \xi_j) \Big)
.\]
In the  penultimate inequality we used inequality $1-x \leq e^{-x}$ for $x \geq 0$. In the ultimate inequality we used $x/2 \leq 1-e^{-x}$ for $0 \leq x \leq 1/2$. This concludes the proof of Lemma \ref{lem:3.4} . \end{proof}
\section{Introduction of new multipliers}
Now we will define a pair of Fourier multipliers, whose corresponding maximal operators are bounded independently of the dimension by a general theory. Then we will try to approximate $s_t$ pointwise by these multipliers.
\begin{Defn}
    \label{def:4.1} For any $\xi \in \T ^d$ and any $t \in \N _0, t \leq d$ we define
    \begin{equation}
        \lambda_t^1(\xi)= \exp\Big(-\frac{t}{d} \sum_{i=1}^d \sin^2(\pi \xi)\Big),
    \end{equation}
   \begin{equation}
        \lambda_t^2(\xi)= (-1)^t\exp\Big(-\frac{t}{d} \sum_{i=1}^d \cos^2(\pi \xi)\Big).
    \end{equation} 
\end{Defn}
Proof of the following theorem is a consequence of the theory of symmetric diffusion semigroups, see  \cite[p. 73]{Stein+1970} and \cite[Theorem 2.2]{Bourgain2020}.
\begin{Thm} \label{thm:4.2}
    For any $f \in l^2(\Z ^d)$ the following inequality holds
    \[ \| \sup_{t>0} |\mathcal{F}^{-1}(\lambda_t^1 \widehat{f})| \|_{l^2(\Z^d)} \leq 2 \|f \|_{l^2(\Z^d)}.\]
\end{Thm}
Notice that $\lambda_t^2(\xi)=(-1)^t\lambda_t^1(\xi+\frac{1}{2})$, where $\xi+\frac{1}{2}=(\xi_1 + \frac{1}{2},..., \xi_d+ \frac{1}{2})$. Because of that we have
\[
\sup_{\|f \|_{l^2(\Z ^d)}=1} \| \sup_{t \in \N, t \leq d} |\mathcal{F}^{-1}(\lambda_t^2 \widehat{f})| \|_{l^2(\Z^d)}= \sup_{\|f \|_{l^2(\Z ^d)}=1} \| \sup_{t \in \N, t \leq d} |\mathcal{F}^{-1}(\lambda_t^1 \widehat{f})| \|_{l^2(\Z^d)}.
\]
We will approximate $s_t(\xi)$ by $\lambda_t^1(\xi)$ or $\lambda_t^2(\xi)$ depending on which term in the right hand side of inequality \eqref{eq:3.2} we expect to be bigger. Define
\[ \Xi_1= \lbrace \xi \in \T^d : \sum_{i=1}^d \cos^2(\pi \xi_i) \geq d/2 \rbrace, \ \ \Xi_2= \T^d \setminus \Xi_1. \]
Then we have the following crucial inequalities.
\begin{Lem} \label{lem:4.3}
For every $t,d \in \N _0$ such that $t \leq d/2$ and every $\xi \in \T ^d$ we have
\begin{enumerate}
    \item if $\xi \in \Xi_1$, then
    \begin{equation}
        |s_t(\xi)- \lambda_t^1(\xi)| \leq 3 \min \Big\lbrace \exp\Big(- \frac{ct}{2d} \sum_{i=1}^d \sin^2( \pi \xi_i) \Big), \frac{t}{d}\sum_{i=1}^d \sin^2( \pi \xi_i) \Big\rbrace, \label{eq:4.3}
        \end{equation}
    \item if $\xi \in \Xi_2$, then
    \begin{equation}
        |s_t(\xi)- \lambda_t^2(\xi)| \leq 3 \min \Big\lbrace \exp\Big(- \frac{ct}{2d} \sum_{i=1}^d \cos^2( \pi \xi_i) \Big), \frac{t}{d}\sum_{i=1}^d \cos^2( \pi \xi_i) \Big\rbrace. \label{eq:4.4}
    \end{equation}
\end{enumerate}
\end{Lem}
\begin{proof} Take $\xi \in \Xi_1$. Then
\[ \sum_{i=1}^d \cos^2(\pi \xi_i) \geq   \sum_{i=1}^d \sin^2(\pi \xi_i).\]
Using Lemma \ref{lem:3.4} and the above, we obtain
\[ |s_t( \xi)| \leq \exp\Big(- \frac{ct}{2d} \sum_{i=1}^d \sin^2( \pi \xi_i) \Big)+\exp\Big(- \frac{ct}{2d} \sum_{i=1}^d \cos^2( \pi \xi_i) \Big) \]
\[
\leq 2 \exp\Big(- \frac{ct}{2d} \sum_{i=1}^d \sin^2( \pi \xi_i) \Big),\]
so that
\[|s_t( \xi)- \lambda_t^1(\xi)| \leq 3 \exp\Big(- \frac{ct}{2d} \sum_{i=1}^d \sin^2( \pi \xi_i) \Big). \]
On the other hand using Lemma \ref{lem:3.1} and inequality $1-e^{-x} \leq x$ we obtain
\[|s_t( \xi)- \lambda_t^1(\xi)| \leq |s_t( \xi)-1| + |1- \lambda_t^1(\xi)| \leq 3 \frac{t}{d} \sum_{i=1}^d \sin^2( \pi \xi_i).\]
This proves \eqref{eq:4.3}. \par
Now let $\xi \in \Xi_1$. Then $\xi+\frac{1}{2}=(\xi_1 + \frac{1}{2},..., \xi_d+ \frac{1}{2}) \in \Xi_2$,
$\lambda_t^2(\xi)=(-1)^t \lambda_t^1(\xi+\frac{1}{2})$ and
\[
s_t(\xi)= \frac{1}{|S_t|} \sum_{x \in S_t} e(x \cdot \xi)=  \frac{1}{|S_t|} \sum_{x \in S_t} (-1)^{\sum_{i=1}^d x_i} e(x \cdot (\xi+\frac{1}{2}))
\]
\[
=\frac{1}{|S_t|} \sum_{x \in S_t} (-1)^{\sum_{i=1}^d |x_i|} e(x \cdot (\xi+\frac{1}{2}))= (-1)^t s_t(\xi+\frac{1}{2}).
\]
By \eqref{eq:4.3} we get
\[
|s_t(\xi)-\lambda_t^2(\xi)|=|s_t(\xi+\frac{1}{2})-\lambda_t^1(\xi+\frac{1}{2})| \]
\[\leq 3 \min \Big\lbrace \exp\Big(- \frac{ct}{2d} \sum_{i=1}^d \sin^2( \pi \xi_i+\frac{\pi}{2}) \Big), \frac{t}{d}\sum_{i=1}^d \sin^2( \pi \xi_i+\frac{\pi}{2}) \Big\rbrace
\]
\[
= 3 \min \Big\lbrace \exp\Big(- \frac{ct}{2d} \sum_{i=1}^d \cos^2( \pi \xi_i) \Big), \frac{t}{d}\sum_{i=1}^d \cos^2( \pi \xi_i) \Big\rbrace.
\]

\end{proof}
\section{Conclusions and proof of the main theorem}
We are almost ready to prove the main theorem. The next lemma is simple, however we give a proof to calculate the explicit constant $\frac{5}{3}$.
\begin{Lem}
    \label{lem:5.1} 
    For any $x \in \mathbb{R}_+$ we have
    \[ \sum_{k=0}^{\infty} \min \lbrace 4^kx, (4^kx)^{-1} \rbrace \leq \frac{5}{3} \]
\end{Lem}
\begin{proof}
    If $x>1$ then 
    \[ \sum_{k=0}^{\infty} \min \lbrace 4^kx, (4^kx)^{-1} \rbrace= \sum_{k=0}^{\infty} 4^{-k} x^{-1}=\frac{4}{3}x^{-1}<\frac{4}{3}. \] 
    If $x \leq 1$, then let $k \in \N_0$ be such that $4^kx=s \in ( \frac{1}{4},1]$. This gives us
    \[
    \sum_{n=0}^{\infty} \min \lbrace 4^nx, (4^nx)^{-1} \rbrace= \sum_{n=0}^k 4^nx+ \sum_{n=k+1}^{\infty} (4^{n}x)^{-1}
    \]
    \[
    = \sum_{n=0}^k 4^{n-k}s+ \sum_{n=k+1}^{\infty} (4s)^{-1} \cdot 4^{k+1-n} \leq (s+ \frac{1}{4s}) \cdot \sum_{n=0}^{\infty} 4^{-n}=\frac{4}{3} (s+\frac{1}{4s}) \leq \frac{5}{3},
    \]
    above we used the fact that function $(\frac{1}{4},1] \ni s \mapsto s+\frac{1}{4s}$ is bounded by $\frac{5}{4}$.
\end{proof}
We are finally ready to prove Theorem \ref{thm:1.3}.
\begin{proof}
    Take any $f \in l^2(\Z ^d)$ we have that
    \[
    \| \sup_{t \in \D, t \leq \sqrt{d}} |\mathcal{M}_t f| \|_{l^2(\Z ^d)} \leq \| \sup_{t \in \D, t \leq \sqrt{d}} |(\mathcal{M}_t-\mathcal{S}_t)f| \|_{l^2(\Z ^d) } + \| \sup_{t \in \D, t \leq \sqrt{d}} |\mathcal{S}_tf| \|_{l^2(\Z ^d) }.
    \]
   By Lemma \ref{lem:2.6} the first term above is bounded by $6\|f \|_{l^2}$. For the second term let $f=f_1+f_2$ with $\widehat{f_1}(\xi)=\widehat{f}(\xi) \mathds{1}_{\Xi_1}(\xi)$ and  
   $\widehat{f_2}(\xi)=\widehat{f}(\xi) \mathds{1}_{\Xi_2}(\xi)$. Then we have
   \[
   \| \sup_{t \in \D, t \leq \sqrt{d}} |\mathcal{S}_tf| \|_{l^2(\Z ^d) }= \| \sup_{t \in \D, t \leq \sqrt{d}} |\mathcal{F}^{-1}(s_t\widehat{f})| \|_{l^2(\Z ^d) }  \]
   \[ \leq
   \| \sup_{t \in \D, t \leq \sqrt{d}} |\mathcal{F}^{-1}(s_t\widehat{f_1})| \|_{l^2(\Z ^d) }
   + \| \sup_{t \in \D, t \leq \sqrt{d}} |\mathcal{F}^{-1}(s_t\widehat{f_2})| \|_{l^2(\Z ^d) } 
   \]
   \[ \leq
    \sum_{i=1}^2 \| \sup_{t \in \D, t \leq \sqrt{d}} |\mathcal{F}^{-1}(\lambda_t^i\widehat{f_i})| \|_{l^2(\Z ^d) }+
    \sum_{i=1}^2 \| \sup_{t \in \D, t \leq \sqrt{d}} |\mathcal{F}^{-1}\big((s_t-\lambda_t^i)\widehat{f_i}\big)| \|_{l^2(\Z ^d) }.
   \]
    By Theorem \ref{thm:4.2} the first sum is bounded  by
    \[ 2\| f_1 \|_{l^2(\Z ^d) }+2\| f_2 \|_{l^2(\Z ^d) } \leq 2\sqrt{2} \big( \| f_1 \|_{l^2(\Z^d)}^2 + \|f_2 \|_{l^2(\Z^d)}^2 \big)^{1/2}= 2\sqrt{2} \| f \|_{l^2(\Z^d)}. \]
    To finish the proof it remains to estimate for $i=1,2$ the following term
    \[
     \| \sup_{t \in \D, t \leq \sqrt{d}} |\mathcal{F}^{-1}\big((s_t-\lambda_t^i)\widehat{f_i}\big)| \|_{l^2(\Z ^d) }.
    \]
    We will only bound the above for $i=1$, the proof for $i=2$ is basically the same.
    From Lemma \ref{lem:4.3} we have that if $\xi \in \Xi_1$, then
    \[
    |s_t(\xi)-\lambda_t^1(\xi)| \leq 3 \min \Big\lbrace \exp\Big(- \frac{ct}{2d} \sum_{i=1}^d \sin^2( \pi \xi_i) \Big), \frac{t}{d}\sum_{i=1}^d \sin^2( \pi \xi_i) \Big\rbrace 
    \]
    \[ \leq
3\min \Big\lbrace \Big(\frac{ect}{2d} \sum_{i=1}^d \sin^2( \pi \xi_i) \Big)^{-1}, \frac{t}{d}\sum_{i=1}^d \sin^2( \pi \xi_i) \Big\rbrace \]
    \[
    \leq \frac{6}{ec} \min \Big\lbrace \Big(\frac{t}{d} \sum_{i=1}^d \sin^2( \pi \xi_i) \Big)^{-1}, \frac{t}{d}\sum_{i=1}^d \sin^2( \pi \xi_i) \Big\rbrace.
    \]
    Using Lemma $\ref{lem:5.1}$ for any $\xi \in \Xi_1$ we obtain
    \[
    \sum_{t \in \D, t \leq \sqrt{d}} |s_t(\xi)-\lambda_t^1(\xi)|^2 \leq \Big( \frac{6}{ec} \Big)^2 \cdot
    \sum_{t \in \D, t \leq \sqrt{d}} \min \Big\lbrace t^{-2}\Big(\frac{1}{d} \sum_{i=1}^d \sin^2( \pi \xi_i) \Big)^{-2}, t^2 \Big(\frac{1}{d}\sum_{i=1}^d \sin^2( \pi \xi_i) \Big)^2 \Big\rbrace 
    \]
    \[
    \leq \Big( \frac{6}{ec} \Big)^2 \cdot \sum_{k=0}^{\infty} \min \Big\lbrace 4^{-k} \Big(\frac{1}{d} \sum_{i=1}^d \sin^2( \pi \xi_i) \Big)^{-2}, 4^k \Big(\frac{1}{d}\sum_{i=1}^d \sin^2( \pi \xi_i) \Big)^2 \Big\rbrace \leq \Big( \frac{6}{ec} \Big)^2 \cdot \frac{5}{3}.
    \]
Finally the following is true
    \[
    \| \sup_{t \in \D, t \leq \sqrt{d}} |\mathcal{F}^{-1}\big((s_t-\lambda_t^1)\widehat{f_1}\big)| \|_{l^2(\Z ^d) } \leq \| \Big( \sum_{t \in \D, t \leq \sqrt{d}} |\mathcal{F}^{-1}\big((s_t-\lambda_t^1)\widehat{f_1}\big)|^2 
 \Big)^{1/2}\|_{l^2(\Z ^d) }
    \]
    \[=
     \Big( \sum_{x \in \Z ^d} \sum_{t \in \D, t \leq \sqrt{d}} |\mathcal{F}^{-1}\big((s_t-\lambda_t^1)\widehat{f_1}\big)(x)|^2 
 \Big)^{1/2}=
 \Big(  \sum_{t \in \D, t \leq \sqrt{d}} \sum_{x \in \Z ^d} |\mathcal{F}^{-1}\big((s_t-\lambda_t^1)\widehat{f_1}\big)(x)|^2 
 \Big)^{1/2}
    \]
    \[=
    \Big(  \sum_{t \in \D, t \leq \sqrt{d}} \| |\mathcal{F}^{-1}\big((s_t-\lambda_t^1)\widehat{f_1}\big)| \|_{l^2(\Z ^d)}^2 
 \Big)^{1/2} \]
 \[=
 \Big(  \sum_{t \in \D, t \leq \sqrt{d}}  \int_{\T^d} |s_t(\xi)-\lambda_t^1(\xi)|^2 |\widehat{f}(\xi)|^2 \cdot\mathds{1}_{\Xi_1}(\xi) d \xi
 \Big)^{1/2}\]
 \[ \leq \frac{6}{ec} \cdot \sqrt{\frac{5}{3}}
 \Big(    \int_{\T^d}  |\widehat{f}(\xi)|^2 \cdot\mathds{1}_{\Xi_1}(\xi) d \xi \Big)^{1/2} 
    = \frac{6}{ec} \cdot \sqrt{\frac{5}{3}} \| \widehat{f_1} \|_{L^2(\T^d)}= \frac{6}{ec} \cdot \sqrt{\frac{5}{3}} \| f_1 \|_{l^2(\Z ^d)}.
    \]
    Similarly we obtain
    \[
     \| \sup_{t \in \D, t \leq \sqrt{d}} |\mathcal{F}^{-1}\big((s_t-\lambda_t^2)\widehat{f_2}\big)| \|_{l^2(\Z ^d) } \leq \frac{6}{ec} \cdot \sqrt{\frac{5}{3}} \| f_2 \|_{l^2(\Z ^d)}.
    \]
    Which gives us
    \[
    \sum_{i=1}^2 \| \sup_{t \in \D, t \leq \sqrt{d}} |\mathcal{F}^{-1}\big((s_t-\lambda_t^i)\widehat{f_i}\big)| \|_{l^2(\Z ^d) } \leq \]
    \[
    \frac{6}{ec} \cdot \sqrt{\frac{5}{3}} \cdot \Big( \| f_1 \|_{l^2(\Z ^d)} + \| f_2 \|_{l^2(\Z ^d)} \Big) \leq \frac{6}{ec} \cdot \sqrt{\frac{10}{3}} \|f \|_{l^2(\Z ^d)}.
    \]
    Combining all of the estimates we finally get
    \[
    \| \sup_{t \in \D, t \leq \sqrt{d}} |\mathcal{M}_t f| \|_{l^2(\Z ^d)} \leq \| \sup_{t \in \D, t \leq \sqrt{d}} |(\mathcal{M}_t-\mathcal{S}_t)f| \|_{l^2(\Z ^d) } + \| \sup_{t \in \D, t \leq \sqrt{d}} |\mathcal{S}_tf| \|_{l^2(\Z ^d) }
    \]
    \[
    \leq \| \sup_{t \in \D, t \leq \sqrt{d}} |(\mathcal{M}_t-\mathcal{S}_t)f| \|_{l^2(\Z ^d) }+ \sum_{i=1}^2 \| \sup_{t \in \D, t \leq \sqrt{d}} |\mathcal{F}^{-1}(\lambda_t^i\widehat{f_i})| \|_{l^2(\Z ^d) } \]
    \[
    +
    \sum_{i=1}^2 \| \sup_{t \in \D, t \leq \sqrt{d}} |\mathcal{F}^{-1}\big((s_t-\lambda_t^i)\widehat{f_i}\big)| \|_{l^2(\Z ^d) }
    \]
    \[
    \leq \|f \|_{l^2(\Z ^d)} \Big( 6+ 2\sqrt{2}+ \frac{6}{ec} \cdot \sqrt{\frac{10}{3}} \Big) \leq 40 \|f \|_{l^2(\Z ^d)}.
    \] 
    This completes the proof of Theorem \ref{thm:1.3}.
\end{proof}

\subsection*{Acknowledgements}
Author was supported by National Science Centre, Poland, grant Sonata Bis nr. 2022/46/E/ST1/00036. This paper constitutes author's master thesis. The author is grateful to his advisor Błażej Wróbel for suggesting the topic and helpful remarks during the preparation of this paper. The author is also grateful to the referee for careful reading of the paper and pointing out that arguments in the section 4 can be simplified.

\end{document}